\documentclass[11pt]{article}

\usepackage{amsfonts}
\usepackage{amssymb,amsmath,amsthm, enumerate, graphicx}
\usepackage{latexsym}
\usepackage{fullpage}
\usepackage{hyperref}

\newtheorem{theorem}{Theorem}

\newtheorem{prop}[theorem]{Proposition}
\newtheorem{question}{Question}

\newtheorem{lemma}[theorem]{Lemma}

\theoremstyle{definition}

\newcounter{tenumerate}

\def\P{\mathbb{P}}

\newcommand{\deq}{\stackrel{\scriptscriptstyle\triangle}{=}}

\renewcommand{\epsilon}{\varepsilon}

\newcommand{\E}{{\mathbb E}}
\newcommand{\remove}[1]{}

\renewcommand{\leq}{\leqslant}
\renewcommand{\geq}{\geqslant}

\def\XXint#1#2#3{{\setbox0=\hbox{$#1{#2#3}{\int}$}
\vcenter{\hbox{$#2#3$}}\kern-.5\wd0}}

\begin{document}

\title{Sensitivity of mixing times}
\author{Jian Ding \thanks{Partially supported by NSF grant DMS-1313596.}\\
University of Chicago  \and Yuval Peres
\\Microsoft Research}

\date{}

\maketitle

\begin{abstract}
In this note, we demonstrate an instance of bounded-degree graphs of size $n$, for which the total variation mixing time for the random walk is decreased by a factor of
$\log n/ \log\log n$ if we multiply the edge-conductances by bounded factors in a certain way.
\end{abstract}
\section{Introduction}
Given a network $G = (V, E)$, where each edge $e\in E$ is endowed with a conductance $c_{u, v}>0$ (the default choice for $c_{u, v}$ is 1), a (lazy) random walk on $G$ repeatedly does the following: when the current state is $v\in V$,  the random walk will stay at $v$ with probability $1/2$ and move to vertex $u$ with probability $c_{u, v}/(2\sum_{w\sim v} c_{w, v})$ for all $u\sim v$. Let $P^t(\cdot, \cdot)$ be the transition probability for the random walk. Then the total variation mixing time with respect to starting node $v$ is defined by (see \cite{AF, LPW09} for more background)
$$t_{\mathrm{mix}}(G, v) = \min\{t :\|P^t(v, \cdot) - \pi(\cdot)\|_{\mathrm{TV}} \leq \frac{1}{4} \}\,,$$
where $\pi$ is the stationary measure for the random walk. Here $\|\mu - \nu\|_\mathrm{TV} \deq \frac{1}{2}\sum_{x\in \Omega} |\mu(x) - \nu(x)|$ is the total variation distance between distributions $\nu$ and $\mu$ supported on $\Omega$. We define the (total variation) mixing time of the graph $G$ by $t_{\mathrm{mix}}(G) = \max_{v\in V} t_{\mathrm{mix}}(G, v)$.

In this note, we consider the sensitivity of the total variation mixing time when the edge conductances are multiplied by bounded factors. We say a family of graphs $\{G_n\}$ is {\bf robust} if for every constant $C>0$ there exists a constant $K>0$ such that if we multiply the edge conductances by  a factor of at most $C$ on these graphs, the corresponding mixing times are preserved up to a factor of $K$. Our main result is the following theorem.
\begin{theorem}\label{thm-main}
There exists a family of uniformly bounded-degree graphs that is not robust. Furthermore, there exists a sequence of graphs $\{G_n\}$ with maximal degrees bounded by 10 and  $|V(G_n)| = n$ as well as a rule to change the edge-conductance up to a factor of 2, such that the total variation mixing times in $G_n$ will change by a factor of at least $c \log n/\log\log n$, where $c$ is an absolute constant.
\end{theorem}
Our example in the preceding theorem is almost optimal (except for the $\log\log n$ term), due to the well known fact that $\lambda^{-1}\leq t_{\mathrm{mix}(G)} = O(\lambda^{-1} \log \frac{1}{\min_x \pi(x)})$ (where $\pi$ is the stationary measure) and that the spectral gap $\lambda$ will be preserved up to constant under the aforementioned perturbation.  There are numerous works aiming at sharp geometric bounds on mixing times such as the Lov{\'a}sz-Kannan bound (\cite{LK99}) and the Fountoulakis-Reed bound \cite{FR07, FR08}, where an upper bound on the mixing time is derived in terms of the expansion profile of the graph; these bounds on mixing times involving the geometry are robust under the conductance perturbation. Theorem~\ref{thm-main} provides a cautionary note on the possibility of developing geometric bounds on mixing times, and implies that a tight (up to constant) bound for the mixing time of random walks on general weighted graphs cannot be robust under bounded perturbation of the edge conductances.

In contrast with Theorem~\ref{thm-main}, many well-known families of graphs are robust. Robustness of general trees was established in \cite{PS11};
some other examples are collected in the next proposition.
\begin{prop}\label{prop}
The following families of graphs are robust:  (1) Tori $\{\mathbb{Z}_n^d: n\in \mathbb{N}\}$ for every fixed $d\in \mathbb{N}$; (2) Erd\H{o}s-R\'enyi supercritical random graph $\mathcal{G}(n, c/n)$ for a fixed $c>1$; (3) Erd\H{o}s-R\'enyi critical random graph $\mathcal{G}(n, 1/n)$; (4) Hypercubes $\{0, 1\}^n$.
\end{prop}

In view of Proposition~\ref{prop}, it is interesting to study how generally does robustness holds. Specifically, the following questions are open.
\begin{question}
Are transitive graphs robust?
\end{question}
A special and important class of Markov chains is the Glauber dynamics for spin systems. In this case, we could focus on perturbation of the updating rate: that is, instead of selecting a uniform spin to update, one selects each spin with probability of between  $c_1/n$ and $c_2/n$,
where $n$ is the size of the underlying graph and $c_1<c_2$ are positive constants.
\begin{question}
Is Glauber dynamics for spin systems robust under bounded perturbation of the updating rate?
\end{question}
One  case where a positive answer to this is easily obtained is when the Dobrushin contraction condition is satisfied \cite{AH87} (see also Theorem 15.1 in \cite{LPW09}).

Finally, there are other notions of mixing times involving different distances between probability measures, one of which is the $L_\infty$ distance $\sup_{x, y\in V}|\frac{P^{t}(x, y)}{\pi(y)} - 1|$.  The robustness of the $L_\infty$ mixing time was studied in Kozma \cite{Kozma}, where he proved that $L_\infty$ mixing time is preserved up to a factor of $\log\log n$. Our result can be seen as a complementary result to \cite{Kozma}.
Note that geometric bounds on $L_\infty$ mixing times have also been developed. In Morris and Peres \cite{MP05}, it was shown that the Lov{\'a}sz-Kannan bound is also effective for $L_\infty$ mixing time, and \cite{MP05} was improved by Goel, Montenegro and Tetali \cite{GMT06} using spectral profile (as opposed to expansion profile).

\section{Constructions and proofs}

Our construction is based on two simple observations. First, by changing the conductance up to a constant factor, one could drastically change the harmonic measure (this idea was used before by Benjamini \cite{Benjamini91} to study  instability of the Liouville property). Thus, one could then decorate the graph in a subset of vertices to incur a delay for the random walk such that the delay will differ significantly before and after changing the conductances. Second, the hitting time in one graph can always be translated to the mixing time for a larger graph. So the sensitivity of hitting times results in the sensitivity of mixing times (on a larger graph).

We now describe the construction of our graphs. We start with a simpler construction which gives a factor of $\sqrt{\log n}$ for mixing times under a certain perturbation. First take a binary tree $T$ rooted at $o$ of height $K$. We distinguish the two children of a parent by \emph{left} and \emph{right} child. Denote by $L$ and $R$ the collection of left and right children in $T$, respectively. For $i\in [K]:=\{1, \ldots, K\}$, denote by $H_i \subset T$ the collection of vertices in the $i$-th level. For $u, v\in T$, denote by $\Gamma_{u, v}$ the collection of vertices that are in the unique path between $u$ and $v$. Define
$$B = \{w\in T : K/4 \leq |\Gamma_{o, w}| \leq K/2, \mbox{ and } ||\Gamma_{o, w} \cap L| - |\Gamma_{o, w} \cap R|| \leq \sqrt{K}\}$$
to be the collection of \emph{balanced} nodes in $T$.  For every $u\in B$, we attach a 3D torus $B_u$ of volume $K$ to $u$. That is to say, $u\in B_u$ and $B_u \cap T = \{u\}$. Furthermore, $B_u \cap B_w = \emptyset$ for $u \neq w$. Denote by $T^*_K$ the graph obtained at this point. Finally, we attach a 3-regular expander of size $K^2 2^K$ to the set of leaves $H_K$ in $T$. That is to say, we attach an expander $G^* = (V^*, E^*)$ such that $V^* \cap T^*_K = H_K$. We denote by $G_1 = (V_1, E_1)$ the final graph obtained from our construction.

\begin{figure}
\includegraphics[width=4in]{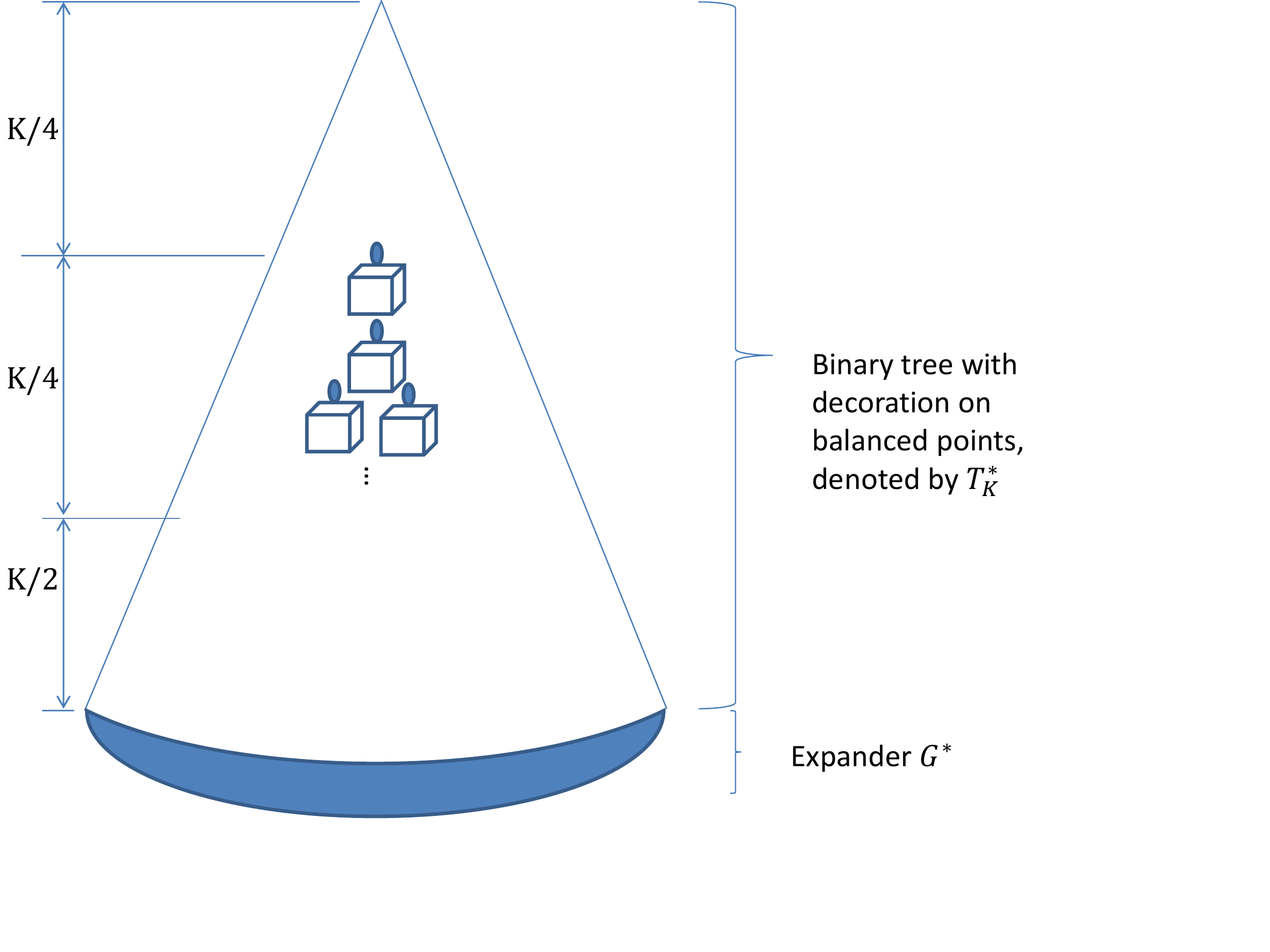}\includegraphics[width=4in]{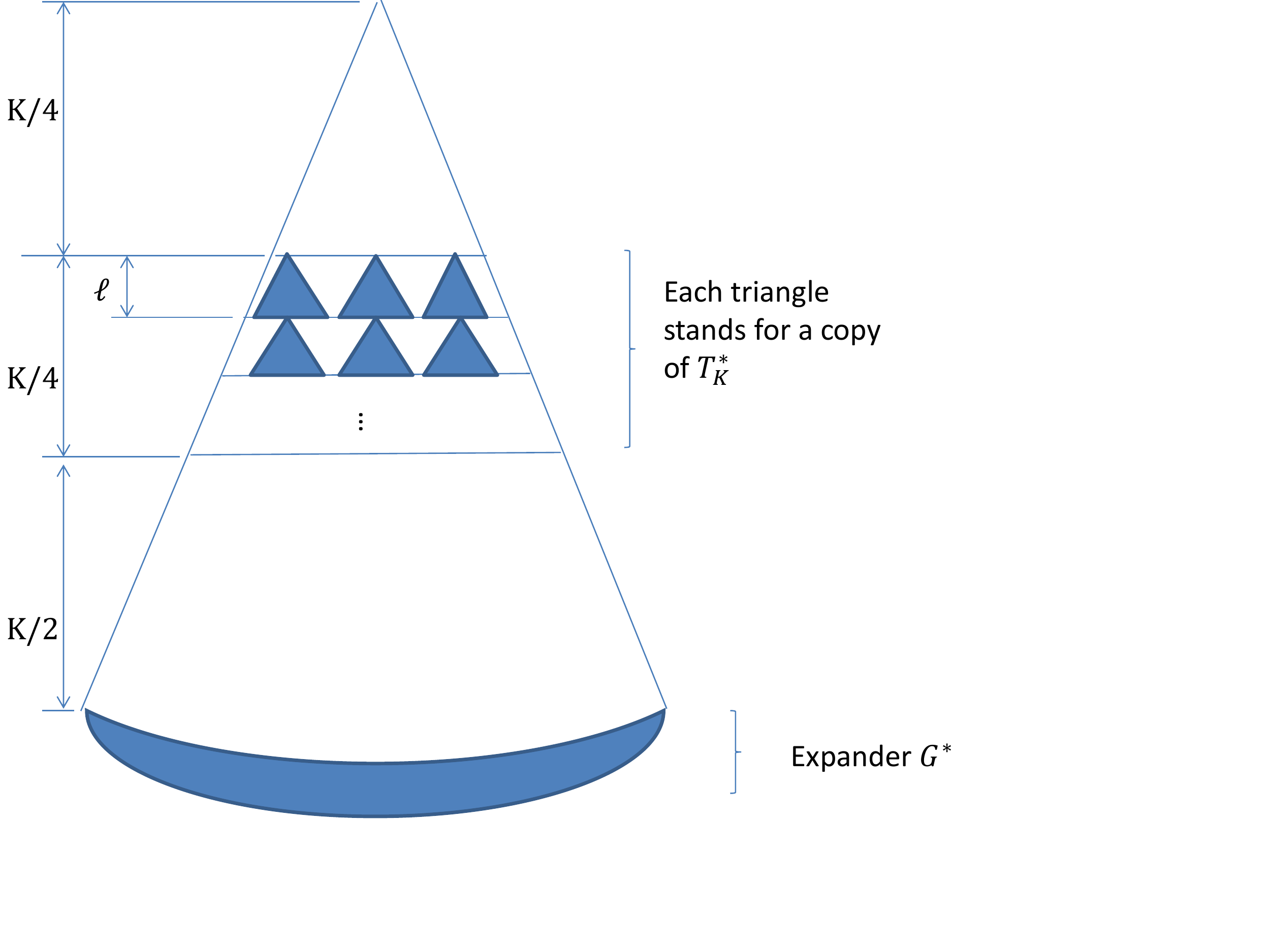}
\caption{On the left, the balanced nodes between levels $K/4$ and $K/2$ in a binary tree of depth $K$ are decorated with 3D tori of volume $K$. This yields $T^*_K$ and we add an expander $G^*$ of volume $K^2 2^K$. On the right, a binary tree of depth $K$ is decomposed into numerous binary trees of depth $\ell$, and those between levels $K/4$ and $K/2$ are replaced by copies of $T^*_\ell$, and then an expander is added as before.}
\label{ree1}
\end{figure}

We will only give a sketch for the mixing time comparison for $G_1$, as a more detailed analysis will be given for a more delicate construction below. It is not hard to see that the mixing time on $G_1$ should be at least the hitting time from the root to the leaf set in $T$. This is because when the random walk on $T$ travels down from the root it typically visits at least $K/100$ balanced points, the attached 3D torus will cause a delay on the hitting time by a factor of $K$. That is to say, the hitting time is at least of order $K^2$, and so is the mixing time. Now, if we change the conductance on every edge $(u, v)\in T$ with $v\in L$ to 2, the mixing times will still be governed by the maximal hitting time to the leaf set. Also, the starting point which maximizes the aforementioned hitting time should be a node in a 3D torus attached to a balanced node deep inside the tree, such that the random walk on $T$ will typically spend time $O(\sqrt{K})$ at balanced nodes before reaching $H_K$. Thus the hitting time of $H_K$ in $G_1$ is $O(K^{3/2})$ (so is the mixing time). This shows that the mixing times will differ by a factor of order $\sqrt{\log |V_1|}$, completing the sketch.

\medskip

\noindent{\bf Main construction.} We construct the family of graphs which achieves the factor of $\log n/\log\log n$ as stated in Theorem~\ref{thm-main}. Again take a binary tree $T$ rooted at $o$ of height $K$. Write $\ell = 100\log K$, and define $\mathcal{H} = \cup_{j=K/4\ell}^{K/2\ell} H_{j\ell}$. For $v\in \mathcal{H}$, let $T_v$ be the complete binary subtree of $T$ rooted at $v$ of height $\ell$. For each $v\in \mathcal{H}$, define
$$A_v = \{w\in T_v: \ell/4 \leq |\Gamma_{v, w}| \leq \ell/2, \mbox{ and } ||\Gamma_{v, w} \cap L| - |\Gamma_{v, w} \cap R|| \leq \sqrt{\ell}\}$$
to be the collection of balanced nodes in the subtree $T_v$. Denote by $\mathcal{A} = \cup_{v\in \mathcal{H}} A_v$. For every $u\in \mathcal{A}$, we attach a 3D torus $B_u$ of volume $K$ to $u$. Finally, we attach a 3-regular expander $G^* = (V^*, E^*)$ of size $K^2 2^K$ to the set of leaves $H_K$ in $T$. We denote by $G = (V, E)$ the final graph obtained from our construction. Note that $|V| = (1+O(1/K))(K^2 2^K)$ (with room to spare in $O(1/K)$).

Now we specify a rule of changing the edge conductance in $G$ in order to obtain a new network supported on the same edge set as $G$. For every edge $(u, v)\in T$ where $u$ is a parent of $v$, we let the new edge conductance $c_{u, v} = 2$ if $v\in L$. The conductance of all the other edges in $G$ is preserved. We denote by $\tilde G = (V, C)$ this weighted graph (or network). At this point, we completed the construction of our examples, and it remains to verify the mixing times on $G$ and $\tilde G$ differ by a desired factor.

\begin{lemma}\label{lem-original}
The mixing time on $G$ satisfies $t_{\mathrm{mix}}(G) \geq c K^2$ for a certain constant $c>0$.
\end{lemma}
\begin{proof}
Let $\tau$ be the hitting time to the set $V^*$ for a random walk $(S_t)$ on $G$. We wish to bound $\E_o \tau$ from below. Let $G'$ be the graph $G$ without attached 3D tori. First, we consider a random walk $(S'_t)$ on $G'$. Denote by $N_i$ the number of times that $(S'_t)$ visits $H_i \cap \mathcal{A}$ before hitting $V^*$, and denote by $\mathcal{N} = \sum_i N_i$. Clearly $(S'_t)$ visits at least one vertex in every $H_i$ before reaching $H_K$. By symmetry, we obtain that
$$\E_o \mathcal{N} = \sum_{i=1}^K \E_oN_i \geq \sum_{i=1}^K \frac{|\mathcal{A} \cap H_i|}{|H_i|} \geq c K\,,$$
for a certain constant $c>0$. Note that $(S_t)$ can be decomposed to $(S'_t)$ and the excursions performed in the attached 3D tori. Recall a standard fact that the expected return time to the origin of a random walk is the total volume normalized by the degree of the origin. Thus, every time $(S_t)$ visits some $v\in \mathcal{A}$, the average time it takes the random walk for the excursion in $B_v$ is at least $K/2$ (and the length of such excursions are independent of the random walk on $G'$). Therefore, we have $\E_o \tau \geq \E_o \mathcal{N} \cdot K/2 \geq cK^2/2$. Let $v^*$ be such that $\E_{v^*} \tau = \max_{v\in V} \E_v \tau$. Thus by Markov property, we have
$$\E_{v^*} \tau \leq c K^2/20 + \P_{v^*}(\tau \geq cK^2/20) \E_{v^*} \tau\,,$$
which yields that
$$\P_{v^*}(\tau \geq cK^2/20) \geq 3/4\,.$$
Since the stationary measure on $V^*$ satisfies that $\pi(V^*) \geq 1 - 1/K \geq 9/10$ for large enough $K$, we obtain that $t_{\mathrm{mix}}(G,v^*) \geq c K^2/20$, as required.
\end{proof}
We now turn to analyze the mixing time on the network $\tilde {G}$.
\begin{lemma}
The mixing time on $\tilde{G}$ satisfies $t_{\mathrm{mix}}(\tilde G) = O(K \log K)$.
\end{lemma}
\begin{proof}
We first show that $t_{\mathrm{mix}}(\tilde G, v) = O(K)$ for all $v\in V^*$. Let $\tau^*$ be the stopping time at which the random walk has visited $V^*$ for $C K$ times, and denote by $\tau_U$ the hitting time of a $U$ for $U\subseteq V$.  Note that $\P_v(\tau_{H_{\mathrm{top}}} < \tau_{V^*}) \leq \mathrm{e}^{-cK}$ for all $v\in V^*$ and a certain $c>0$, where $H_{\mathrm{top}} = \cup_{i=1}^{3K/4} H_i$. This implies that the random walk will not hit $\mathcal{A}$ before $\tau^*$ except with exponentially small probability. Also, on the rare event that the random walk does hit $\mathcal{A}$, it only increases the stopping time $\tau^*$ by an additive term up to $O(K^3)$ on average. Thus, we have $\E_v \tau^* = O(K)$. Next, denoting by $\tilde V^* = V^* \cup \cup_{i=3K/4}^K H_i$, we see that the random walk stays within $\tilde V^*$ with probability at least $1 - \mathrm{e}^{-cK}$. In addition, the majority (above $9/10$ when $K$ is large) of the stationary measure is supported on $V^*$ (and thus also on $\tilde V^*$). Therefore, the mixing time for the whole network is bounded by the mixing time of the induced sub-graph $\tilde G^*$ on $\tilde V^*$ from above up to a multiplicative constant. Furthermore, since the induced subgraph $\tilde G^*$ is also an expander,  we have $t_{\mathrm{mix}}(\tilde G^*) \leq CK$ for a finite constant $C$. Altogether, we have shown $t_{\mathrm{mix}}(\tilde G, v) = O(K)$ uniformly for all $v\in V^*$.

Denote again by $\tau$ the hitting time to $V^*$. In light of the above discussion, it remains to bound $\max_v\E_v \tau$. Analogous to the proof of Lemma~\ref{lem-original}, we consider a random walk $(\tilde S'_t)$ on $\tilde{G}'$, which is the network obtained from $\tilde {G}$ by ignoring the attached 3D tori. Suppose $(\tilde S'_t)$ started at some $v\in H_k\cap T_w$ where $w\in \mathcal{H}$. Note that before hitting $V^*$, the expected number of times that $(\tilde S'_t)$ visits $H_{k-j}$ is $O(\mathrm{e}^{-cj})$ for a certain constant $c>0$ (as the random walk is biased toward the leaves). Therefore, the expected number of times that $(\tilde S'_t)$ visits $\cup_{i=1}^k H_i$ is $O(1)$. Also clearly, the expected number of visits to each $H_i$ is $O(1)$, thus the expected number of visits to $|T_w |$ is $O(\log K)$. Next, for $u\in \mathcal{H} \setminus \cup_{i=1}^k H_i$, we try to bound $N_u$, which is the number of visits to $T_u \cap \mathcal{A}$ before $\tau$. Observe that $\E_v N_u = \P_v(\tau_u < \tau) \E_u(N_u)$. In addition, by a simple application of large deviation principle,
$$\E_u N_u \leq O(1) \sum_{i=\ell/4}^{\ell} \P(|Z_i|\leq \sqrt{\ell}) = O(K^{-2})\,.$$
where $Z_i$ is a sum of $i$ independent Bernoulli variables (taking values $\pm 1$) which has bias $1/6$. Therefore, we obtain that (denoting by $\mathcal{N}$ the total number of visits to $\mathcal{A}$ before $\tau$)
$$\E_v \mathcal{N} \leq O(\log K) + \sum_{u\in \mathcal{H}\setminus \{w\}} \E_v N_u \leq O(\log K) + K O(K^{-2}) = O(\log K)\,.$$
This implies that for the random walk on $\tilde {G}$, we have $\E_v \tau = O(K \log K)$ uniformly for all $v\in T$. Since the maximal hitting time for a 3D torus of volume $K$ is $O(K)$, this yields that $\max_{v\in V} \E_v \tau = O(K \log K)$, completing the proof of the lemma together with the fact $t_{\mathrm{mix}}(\tilde G, v^*) = O(K)$ for all $v^*\in V^*$.
\end{proof}
\begin{proof}[Proof of Theorem~\ref{thm-main}]
Combining the preceding two lemmas, we obtain a ratio of order $\log n/\log\log n$ for the mixing times after the perturbation of the conductances, thereby completing the proof of Theorem~\ref{thm-main}.
\end{proof}
 We remark that a slightly more careful analysis would lead to a ratio of order $\log n/\sqrt{\log\log n}$. In addition, by iterating the aforementioned construction, one could obtain a ratio of order $\log n/\log^{(j)} n$ for any fixed $j\in \mathbb{N}$, where $\log^{(j)} n$ is the iterated logarithm of order $j$.

Finally, we provide a proof of Proposition~\ref{prop}, which are simple consequences of known results.
\begin{proof}[\bf Proof of Proposition~\ref{prop}]
 (1) Torus $\mathbb{Z}_n^d$: the upper bound $C_d n^2$ can be deduced from Lov{\'a}sz-Kannan bound (\cite{LK99}) which is robust, and the lower bound $c_d n^2$ is given by the inverse spectral gap, which is also robust. (2) Erd\H{o}s-R\'enyi graph in supercritical case: the upper bound $O(\log^2 n)$ is given by the robust Fountoulakis-Reed bound (\cite{FR07, FR08}, see also \cite{BKW06}) and the lower bound is due to the fact that there is an induced path of length $\Theta(\log n)$. (3) Erd\H{o}s-R\'enyi graph in critical case: the upper bound $O(n)$ is given by the maximal commute time (\cite{NP08}) which is robust, while the proof of the lower bound in \cite{NP08} is also robust.(4) Hypercube: the upper bound $O(n \log n)$ is given by the log-Sobolev constant (see, e.g.,  \cite{DS96}) which is robust. For the lower bound, one could employ the coupon collecting argument. Note that after changing the conductance, the random walk does not update each coordinate in a uniform rate. Instead, each coordinate will be selected to update with probability of order $1/n$. Therefore, there exists a constant $c>0$ such that after $cn\log n$ steps the random walk started at $\underline{0}$ has at least $\frac{n}{2} + n^{2/3}$ 0's with high probability. Since in the stationary measures (for both original hypercube and after perturbation) such configurations have negligible probability, this implies that the mixing time is larger than $cn\log n$.
\end{proof}

\noindent {\bf Acknowledgment.} We are grateful to Riedi Basu, Itai Benjamini, Laura Florescu, Shirshendu Ganguly, Gady Kozma and Jeffrey Steif for helpful comments.


\small
\begin{thebibliography}{10}

\bibitem{AH87}
M.~Aizenman and R.~Holley.
\newblock Rapid convergence to equilibrium of stochastic {I}sing models in the
  {D}obrushin {S}hlosman regime.
\newblock In {\em Percolation theory and ergodic theory of infinite particle
  systems ({M}inneapolis, {M}inn., 1984--1985)}, volume~8 of {\em IMA Vol.
  Math. Appl.}, pages 1--11. Springer, New York, 1987.

\bibitem{AF}
D.~Aldous and J.~Fill.
\newblock {\em Reversible Markov Chains and Random Walks on Graphs}.
\newblock In preparation, available at {\tt
  http://www.stat.berkeley.edu/~aldous/RWG/book.html}.

\bibitem{Benjamini91}
I.~Benjamini.
\newblock Instability of the {L}iouville property for quasi-isometric graphs
  and manifolds of polynomial volume growth.
\newblock {\em J. Theoret. Probab.}, 4(3):631--637, 1991.

\bibitem{BKW06}
I.~Benjamini, G.~Kozma, and N.~C. Wormald.
\newblock The mixing time of the giant component of a random graph.
\newblock Preprint, available at {\tt http://arxiv.org/abs/math/0610459}.

\bibitem{DS96}
P.~Diaconis and L.~Saloff-Coste.
\newblock Logarithmic {S}obolev inequalities for finite {M}arkov chains.
\newblock {\em Ann. Appl. Probab.}, 6(3):695--750, 1996.

\bibitem{FR07}
N.~Fountoulakis and B.~A. Reed.
\newblock Faster mixing and small bottlenecks.
\newblock {\em Probab. Theory Related Fields}, 137(3-4):475--486, 2007.

\bibitem{FR08}
N.~Fountoulakis and B.~A. Reed.
\newblock The evolution of the mixing rate of a simple random walk on the giant
  component of a random graph.
\newblock {\em Random Structures Algorithms}, 33(1):68--86, 2008.

\bibitem{GMT06}
S.~Goel, R.~Montenegro, and P.~Tetali.
\newblock Mixing time bounds via the spectral profile.
\newblock {\em Electron. J. Probab.}, 11:no. 1, 1--26 (electronic), 2006.

\bibitem{Kozma}
G.~Kozma.
\newblock On the precision of the spectral profile.
\newblock {\em ALEA Lat. Am. J. Probab. Math. Stat.}, 3:321--329, 2007.

\bibitem{LPW09}
D.~A. Levin, Y.~Peres, and E.~L. Wilmer.
\newblock {\em Markov chains and mixing times}.
\newblock American Mathematical Society, Providence, RI, 2009.
\newblock With a chapter by James G. Propp and David B. Wilson.

\bibitem{LK99}
L.~Lov{\'a}sz and R.~Kannan.
\newblock Faster mixing via average conductance.
\newblock In {\em Annual {ACM} {S}ymposium on {T}heory of {C}omputing
  ({A}tlanta, {GA}, 1999)}, pages 282--287. ACM, New York, 1999.

\bibitem{MP05}
B.~Morris and Y.~Peres.
\newblock Evolving sets, mixing and heat kernel bounds.
\newblock {\em Probab. Theory Related Fields}, 133(2):245--266, 2005.

\bibitem{NP08}
A.~Nachmias and Y.~Peres.
\newblock Critical random graphs: diameter and mixing time.
\newblock {\em Ann. Probab.}, 36(4):1267--1286, 2008.

\bibitem{PS11}
Y.~Peres and P.~Sousi.
\newblock Mixing times are hitting times of large sets.
\newblock Preprint, available at {\tt http://arxiv.org/abs/1108.0133}.

\end{thebibliography}
\end{document}